\documentclass[smallextended]{svjour3}       

\smartqed

\usepackage{graphicx}
\usepackage{amssymb}
\usepackage{amsmath}
\usepackage{mathtools}
\usepackage{color}
\usepackage{cases}
\usepackage[left=2cm, right=2cm]{geometry}

\newcommand{\g}{\mathfrak{g}}
\newcommand{\J}{\mathcal{J}}

\begin{document}

\title{Classification of integrable complex structures on 6-dimensional product Lie algebras}

\titlerunning{Integrable complex structures on 6-dimensional product Lie algebras}

\author{Andrzej Czarnecki}

\institute{Andrzej Czarnecki \at Jagiellonian University, \L{}ojasiewicza 6, 30-348 Krakow, Poland\\\email{andrzejczarnecki01@gmail.com}}


\maketitle

\begin{abstract}
We classify all integrable complex structures on 6-dimensional Lie algebras of the form $\g\times\g$.
\end{abstract}

\section{Introduction.}

In this paper we finish the classification of integrable complex structures on 6-dimensional Lie algebras of the form $\g\times\g$, i.e. automorphisms $\J: \g\times\g \longrightarrow \g\times\g$ such that $\J^2= -id$ and the Nijenhuis tensor
\begin{align*}
N_{\J}(v,w):=[v,w]+\J[\J v,w]+\J[v,\J w]-[\J v,\J w]=0
\end{align*}
vanishes. Such a structure on the Lie algebra is of course equivalent to a left invariant complex structure on any Lie group $H$ with $T_eH=\g\times\g$.

Classification of integrable complex structures on all real Lie algebras is a well established problem, cf. a summary of results on their existence in \cite{classif}. In dimension 6, the question is settled only for special -- abelian -- complex structures (cf. \cite{andr}), and for nilpotent algebras (cf. \cite{nilmanifolds,sal}). In \cite{jamarcin} we established which 6-dimensional Lie algebras that split as a product $\g \times \g$ admit an integrable complex structure, a problem studied before in the special cases of $\mathfrak{su}(2)\times\mathfrak{su}(2)$ and $\mathfrak{sl}(2,\mathbb{R})\times\mathfrak{sl}(2,\mathbb{R})$ in \cite{mag1,mag2}. In this two thorough papers Magnin describes, among other things, all possible integrable almost complex structures on these algebras, and we provide here a similar description for all other cases.

The main application of \cite{jamarcin} was to generalize the notion of normal structures. Consider for example an almost contact structure: a smooth Riemannian manifold $M$ with a vector field $\xi$ and an almost complex structure $\phi$ on its orthogonal complement, $\xi^\top$. We say that this structure is \emph{normal} if the complex structure $\phi\oplus\left[\begin{smallmatrix} 0 & 1 \\ -1 & 0\end{smallmatrix}\right]$ on $T\left(M\times \mathbb{R}\right)=\xi^\top\oplus\xi\oplus T\mathbb{R}$ extending $\phi$ is integrable (cf. \cite{jamarcin} and references therein for more examples of similar definitions and viewpoints in various subjects of geometry). It is natural to want to replace the $\mathbb{R}$-action by an action of arbitrary group $G$, however one runs into the problem of ambiguity of definition: there can be essentially only one complex structure on $\xi\oplus T\mathbb{R}$, but there may be many distinct structures on $\g\times \g$, possibly sufficiently different between different groups to prevent a definition of the structure on the product $M\times G$ independent from the particularities of the group in question. Surprisingly, it turns out that each 3-dimensional group admits a complex structure of a very special type on its square $G\times G$ (Theorem 1 and Proposition 2 in \cite{jamarcin}, cf. Remark \ref{mix}) which allowed us to define in \cite{MAR} \emph{mixed normal structures} for smooth manifolds with a locally free actions of any 3-dimensional Lie group. We feel that the present classification of all possible integrable complex structures for squares of 3-dimensional Lie groups provides an important starting point for tackling the same problem for actions of groups of higher dimension and associated structures occurring naturally.

Apart from these geometric motivations, the problem of classification and understanding of all 6-dimensional real Lie algebras admitting integrable complex structures is an outstanding and difficult one.

\section{Notations and preliminaries}

Inside the direct product $\g\times\g$ we keep the distinction between two copies of $\g$ by adding asterisks to the second one: any vector decorated with an asterisk is understood to lie in $\g^*=0\times \g$, while those without it lie in $\g=\g\times 0$. The product inherits the bracket operation on each factor from $\g$: $[u+u^*,v+v^*]=[u,v]+[u^*,v^*]$.

We also distinguish the two components of a complex structure $\J$ writing often $\J v=Jv+J^*v$ to indicate its $\g$- and $\g^*$-parts separately.

Since every complex structure will be integrable, we use the convention that every Nijenhuis bracket is equal to zero without further notice.

We recall for reference the well-known classification of 3-dimensional Lie algebras and the main result from \cite{jamarcin}.

\begin{proposition}{\cite{bia}}\label{alg} Let $\{e_1, e_2, e_3\}$ be a basis of $\mathbb{R}^3$. Up to isomorphism of Lie algebras, the following list yields all Lie brackets on $\mathbb{R}^3$
{\renewcommand\labelenumi{(\theenumi)}
\begin{enumerate}
\item $[e_1,e_3]= 0$, $[e_2,e_3]= 0$, $[e_1,e_2]= 0$ 
\item $[e_1,e_3]= 0$, $[e_2,e_3]= 0$, $[e_1,e_2]= e_1$
\item $[e_1,e_3]= 0$, $[e_2,e_3]= 0$, $[e_1,e_2]= e_3$ 
\item $[e_1,e_3]= e_1$, $[e_2,e_3]= \theta e_2$, $[e_1,e_2]= 0$ for $\theta\neq 0$ (the case $\theta = 1 $ is considered to be Bianchi's ninth type)
\item $[e_1,e_3]= e_1$, $[e_2,e_3]= e_1+e_2$, $[e_1,e_2]= 0$ 
\item $[e_1,e_3]= \theta e_1-e_2$, $[e_2,e_3]= e_1+\theta e_2$, $[e_1,e_2]= 0$  for $\theta > 0$
\item $[e_1,e_3]= e_2$, $[e_2,e_3]= e_1$, $[e_1,e_2]=e_3$ 
\item $[e_1,e_3]=-e_2 $, $[e_2,e_3]=e_1 $, $[e_1,e_2]= e_3$ 
\end{enumerate}}
\end{proposition}

We refer to such bases of these algebras (and by extension to bases $\{e_1, e_2, e_3, e_1^*, e_2^*, e_3^*\}$ of $\g\times\g$) as \emph{standard bases}. Note that $Aut(\g\times\g)$ does not preserve standard bases, only $Aut(\g)\times Aut(\g)$ does. We refrain from describing the moduli space of integrable complex structures under the action of full $Aut(\g\times\g)$ since the geometric applications mentioned above justify keeping the two factors of $\g\times\g$ distinct. We also note that these moduli spaces in general do not carry as rich and interesting structure as in the case of $\mathfrak{su}(2)\times\mathfrak{su}(2)$ and $\mathfrak{sl}(2,\mathbb{R})\times\mathfrak{sl}(2,\mathbb{R})$ described in detail in \cite{mag1,mag2}. Magnin proves that the two are complex manifolds, while e.g. the moduli space for the abelian algebra is odd-dimensional.

We write $\g_{n}$ or "algebra of type (n)" for the n-th entry in this list, using $\g_{n}^{\theta}$ where the parameter is needed.


\begin{proposition}{\cite{jamarcin}}\label{main}
A real 6-dimensional Lie algebra of the form $\g \times \g$ (for some real 3-dimensional Lie algebra $\g$) carries an integrable complex structure if $\g$ is of type (1), (2), (3), (6), (7), (8), or (4) with parameter $\theta=1$ in Proposition \ref{alg} above. There is however no such structure for the type (5), and for all other parameters in (4).
\end{proposition}

We will now proceed to describe all integrable complex structures in each case. Each time the result will be given as: if $\J$ is an integrable complex structure on algebra (n), then in some standard base it is of the form prescribed. The method of proof will also be the same (and very simple) each time:

\begin{enumerate}
\item Supposing $\J$ is an integrable complex structure on $\g\times\g$ we derive the restrictions it must obey from the Nijenhuis brackets $N_{\J}(e_1,e_3)$, $N_{\J}(e_2,e_3)$, and $N_{\J}(e_1,e_2)$.
\item We describe different orbits of algebra automorphisms acting on $\g$, and decide which can and can not contain \emph{quasi-invariant} vectors, i.e. such $v\in\g$ that $\J v=\lambda v + J^*v$. We note that these orbits are particularly simple due to the low dimension, and their description will be given without proof.
\item We then check whether $J^*\g$ is 1- or 3-dimensional.
\end{enumerate}

These three simple steps are sufficient to describe integrable complex structures in each case because of the following observations.

\begin{remark}
A quasi-invariant vector must always exist (the characteristic polynomial of $J$ has a real root). Observe that $J^*v$ must then be non-zero, and that $\J J^*v = (-1-\lambda^2)v-\lambda J^*v$.
\end{remark}

\begin{remark}
The space $J^*\g$ is either 1- or 3-dimensional, in order to $\J$-invariant space $\J\g\cap\g$ to be of even dimension. In particular, the former case is equivalent (aside from the abelian case of algebra (1)) to $[J^*e_1,J^*e_2]=[J^*e_2,J^*e_3]=[J^*e_3,J^*e_1]=0$.
\end{remark}

\begin{remark}
If $J^*\g$ is 3-dimensional, then $\{e_1,e_2,e_3,\J e_1,\J e_2, \J e_3\}$ is a base of $\g\times
\g$ and $N_{\J}(e_1,e_3)=N_{\J}(e_2,e_3)=N_{\J}(e_1,e_2)=0$ is equivalent to $N_{\J}\equiv 0$, since $N_{\J}(v,w)= - N_{\J}(\J v,\J w)= \J N_{\J}(\J v,w)= \J N_{\J}(v,\J w)$. If $J^*\g$ is 1-dimensional, then most of the time so will be $J\g^*$ and $N_{\J}(v,w^*)$ will be more or less automatically zero.
\end{remark}

We will elaborate and see in each case that these observations together with restrictions on quasi-invariant vectors are sufficient to give a concise description of $\J$.


\section{Case (1), the abelian algebra}

This is the special case since the Nijenhuis bracket does not give any additional conditions. Every complex structure is thus integrable and takes the simplest possible form
$$
\left[
\begin{array}{cccccc}
0 & 0 & 0 & -1 & 0 & 0\\
0 & 0 & 0 & 0 & -1 & 0\\
0 & 0 & 0 & 0 & 0 & -1\\
1 & 0 & 0 & 0 & 0 & 0\\
0 & 1 & 0 & 0 & 0 & 0\\
0 & 0 & 1 & 0 & 0 & 0
\end{array}
\right]
$$
but not necessarily in a standard basis. For the sake of completeness we note that in some standard basis $\J$ takes either the form
$$
\left[
\begin{array}{cccccc}
X & A & 0 & -1-X^2-AY & -AX-AB & 0\\
Y & B & 0 & -XY-BY & -1-B^2-AY & 0\\
Z & C & \lambda & -XZ-YC-Z\lambda & -AZ-BC-C\lambda & -1-\lambda^2\\
1 & 0 & 0 & -X & -A & 0\\
0 & 1 & 0 & -Y & -B & 0\\
0 & 0 & 1 & -Z & -C & -\lambda
\end{array}
\right]\text{ or }
\left[
\begin{array}{cccccc}
0 & -1 & 0 & 0 & 0 & 0\\
1 & 0 & 0 & 0 & 0 & 0\\
0 & 0 & \lambda & 0 & 0 & -1-\lambda^2\\
0 & 0 & 0 & 0 & -1 & 0\\
0 & 0 & 0 & 1 & 0 & 0\\
0 & 0 & 1 & 0 & 0 & -\lambda
\end{array}
\right]
$$
(for any choice of $X$, $Y$, $Z$, $A$, $B$, $C$, and $\lambda$), as can be easily checked.

\section{Case (2)}

The algebra is given by a multiplication table $[e_1,e_3]= 0$, $[e_2,e_3]= 0$, $[e_1,e_2]= e_1$. Suppose that $\J$ is an integrable complex structure and
\begin{align*}
\J e_1 &=  Xe_1 + Ye_2 + Ze_3 + J^*e_1 \\
\J e_2 &=  Ae_1 + Be_2 + Ce_3 + J^*e_2 \\
\J e_3 &=  Pe_1 + Qe_2 + Re_3 + J^*e_3
\end{align*}
Then the Nijenhuis brackets are
\begin{align*}
N_{\J}(e_1,e_3) &=  [e_1,e_3] + \J[Je_1,e_3]+\J[e_1,Je_3] - [\J e_1,\J e_3] \\
&=0+0+Q\J e_1-XQe_1+YPe_1-[J^*e_1,J^*e_3]\\
N_{\J}(e_2,e_3) &=  [e_2,e_3] + \J[Je_2,e_3]+\J[e_2,Je_3] - [\J e_2,\J e_3] \\
&=0+0+P\J e_1-AQe_1+BPe_1-[J^*e_2,J^*e_3]\\
N_{\J}(e_1,e_2) &=  [e_1,e_2] + \J[Je_1,e_2]+\J[e_1,Je_2] - [\J e_1,\J e_2] \\
&=e_1+X\J e_1+B\J e_1-XBe_1+AYe_1-[J^*e_1,J^*e_2]\\
\end{align*}
and together give
\begin{align*}
Q\J e_1&= (XQ-YP)e_1+[J^*e_1,J^*e_3]\\
P\J e_1&= (AQ-BP)e_1+[J^*e_2,J^*e_3]\\
(X+B)\J e_1&= (-1+XB-AY)e_1+[J^*e_1,J^*e_2]
\end{align*}
The possible candidates for a quasi-invariant vector are of four types.
\begin{proposition}
Each automorphism of algebra (2) leaves directions $span\{e_1\}$ and $span\{e_3\}$ invariant. Each vector is equivalent to either $e_1$, $e_3$, a non-trivial combination of these two ($\alpha e_1+\gamma e_3$, $\alpha\gamma\neq 0$), or $e_2$ under some automorphism.
\end{proposition}
We will show that
\begin{proposition}
The vector $e_3$ is quasi-invariant.
\end{proposition}
Suppose it is not, and so at least one of $P$ and $Q$ is non-zero, and hence $e_1$ is quasi-invariant. Thus $Y=Z=0$ and $X=\lambda$ and from the last equation we have
$$
(\lambda+B)\lambda=(-1+\lambda B)
$$
which gives $\lambda^2=-1$, a contradiction.
So $e_3$ must be quasi-invariant and also $X+B=0$ (lest $e_1$ be quasi-invariant again) and also $-1-X^2-AY=0$ (to satisfy the last equation). This implies that both $A$ and $Y$ cannot vanish so $e_2$ can not be quasi-invariant as well. It is also a matter of simple computation to see that neither can a non-trivial combination of $e_1$ and $e_3$, but we don't need that now.

We see that $J^*\g_2$ must be one-dimensional, spanned by $J^*e_3$. Thus we have so far
\begin{align*}
\J e_1 &=  Xe_1 + Ye_2 + Ze_3 + \kappa J^*e_3 \\
\J e_2 &=  Ae_1 -Xe_2 + Ce_3 + \tau J^*e_3 \\
\J e_3 &=  \lambda e_3 + J^*e_3
\end{align*}
Without loss of generality we can assume $\tau$ to be zero (because there is an algebra isomorphism fixing $e_1$ and $e_3$ and taking $e_2$ to $e_2-\tau e_3$). In this point we exchange a complicated description of $\J$ in a fixed base for a more elegant description in some base, albeit also a standard one. We proceed to compute
\begin{align*}
-e_1 =\J(Xe_1 + Ye_2 + Ze_3 + \kappa J^*e_3)=&X^2e_1 +XYe_2 +XZe_3 -X\kappa J^*e_3\\&+ AYe_1 -XYe_2 + CYe_3\\&+Z\lambda e_3+ZJ^*e_3\\&+(-1-\lambda^2)\kappa e_3-\lambda\kappa J^*e_3\\
-e_2 =\J(Ae_1 -Xe_2 + Ce_3)=&  AXe_1 + AYe_2 + AZe_3 + A\kappa J^*e_3\\
&- AXe_1 + X^2e_2 -XCe_3\\
&+C\lambda e_3+CJ^*e_3
\end{align*}
that give
$$
\begin{cases}
0=-1-X^2-AY\\
0=CY+XZ+Z\lambda-(1+\lambda^2)\kappa\\
0=X\kappa + Z -\lambda\kappa\\
0=AZ-XC+C\lambda\\
0=A\kappa+C
\end{cases}
\text{\quad or equivalently \quad}
\begin{cases}
0=-1-X^2-AY\\
Z=\kappa(\lambda -X)\\
C=-\kappa A
\end{cases}
$$
And we summarise it in the following way.
\begin{proposition}
Each integrable structure $\J$ on $\g_{2}\times \g_{2}$ is (in some standard basis) of the form
$$
\left[
\begin{array}{cccccc}
X & \frac{-1-X^2}{Y} & 0 & 0 & 0 & 0\\
Y & -X & 0 & 0 & 0 & 0\\
\kappa(\lambda-X) & \kappa \frac{1+X^2}{Y} & \lambda & \kappa^*(-1-\lambda^2) & 0 & -1-\lambda^2\\
0 & 0 & 0 & X^* & \frac{-1-(X^*)^2}{Y^*} & 0\\
0 & 0 & 0 & Y^* & -X^* & 0\\
\kappa & 0 & 1 & \kappa^*(-\lambda-X^*) & \kappa^* \frac{1+(X^*)^2}{Y^*} & -\lambda
\end{array}
\right]
$$
for some numbers $\lambda$, $X$, $X^*$, $\kappa$, $\kappa^*$ and non-zero $Y$ and $Y^*$. Any such choice of numbers provides an integrable complex structure.
\end{proposition}

\section{Case (3), the Heisenberg algebra}

The algebra is given by a multiplication table $[e_1,e_3]= 0$, $[e_2,e_3]= 0$, $[e_1,e_2]= e_3$.
Suppose that $\J$ is an integrable complex structure and
\begin{align*}
\J e_1 &=  Xe_1 + Ye_2 + Ze_3 + J^*e_1 \\
\J e_2 &=  Ae_1 + Be_2 + Ce_3 + J^*e_2 \\
\J e_3 &=  Pe_1 + Qe_2 + Re_3 + J^*e_3
\end{align*}
Then the Nijenhuis brackets are
\begin{align*}
N_{\J}(e_1,e_3) &=  [e_1,e_3] + \J[Je_1,e_3]+\J[e_1,Je_3] - [\J e_1,\J e_3] \\
&=0+0+Q\J e_3-XQe_3+YPe_3-[J^*e_1,J^*e_3]\\
N_{\J}(e_2,e_3) &=  [e_2,e_3] + \J[Je_2,e_3]+\J[e_2,Je_3] - [\J e_2,\J e_3] \\
&=0+0+P\J e_3-AQe_3+BPe_3-[J^*e_2,J^*e_3]\\
N_{\J}(e_1,e_2) &=  [e_1,e_2] + \J[Je_1,e_2]+\J[e_1,Je_2] - [\J e_1,\J e_2] \\
&=e_3+X\J e_3+B\J e_3-XBe_3+AYe_3-[J^*e_1,J^*e_2]\\
\end{align*}
and together give
\begin{align*}
Q\J e_3&= (XQ-YP)e_3+[J^*e_1,J^*e_3]\\
P\J e_3&= (AQ-BP)e_3+[J^*e_2,J^*e_3]\\
(X+B)\J e_3&=(-1+XB-AY)e_3+[J^*e_1,J^*e_2]
\end{align*}
The possible candidates for a quasi-invariant vector are of two types.
\begin{proposition}
Each automorphism of algebra (3) leaves the direction $span\{e_3\}$ invariant. Each vector is equivalent either to $e_1$ or $e_3$ by some automorphism.
\end{proposition}
Again, it is easy to see that the first two equation imply directly that $e_3$ is quasi invariant (since $e_3$ being quasi-invariant means that $P=Q=0$). Thus we have $\J e_3=\lambda e_3+J^*e_3$ and the equations now read
\begin{align*}
0&=[J^*e_1,J^*e_3]\\
0&=[J^*e_2,J^*e_3]\\
(X+B)\J e_3&=(-1+XB-AY)e_3+[J^*e_1,J^*e_2]
\end{align*}
We have now two options.
\begin{enumerate}
\item Either $X+B\neq 0$ and $J^*e_1$, $\frac{1}{X+B}J^*e_2$ and $J^*e_3$ is a standard basis for $\g^*_3$. Then we also have $(X+B)\lambda=-1+XB-AY$ and we compute from $-e_1 =\J(Xe_1 + Ye_2 + Ze_3 + J^*e_1)$ and $-e_2 =\J(Ae_1 + Be_2 + Ce_3+ J^*e_2)$ that
\begin{align*}
\J J^*e_1&=(-1-X^2-AY)e_1 - Y(X+B)e_2-(XZ+CY+Z\lambda)e_3\\
&-XJ^*e_1-YJ^*e_2-ZJ^*e_3\\
\J J^*e_2&=A(X+B)e_1 - (1+B^2+AY)e_2-(AZ+BC+C\lambda)e_3\\
&-AJ^*e_1-BJ^*e_2-CJ^*e_3
\end{align*}
\item Or $X+B=0$ and $im J^*$ is one dimensional, generated by $J^*e_3$. Then we compute from $-e_1 =\J(Xe_1 + Ye_2 + Ze_3 + \kappa J^*e_3)$ and $-e_2 =\J(Ae_1 + -Xe_2 + Ce_3+\tau J^*e_3)$ that give
$$
\begin{cases}
0=-1-X^2-AY\\
0=XZ+CY+Z\lambda+\kappa(-1-\lambda^2)\\
0=X\kappa+Y\tau+Z-\lambda\kappa\\
0=AZ-XC+C\lambda+\tau(-1-\lambda^2)\\
0=A\kappa-X\tau+C-\tau\lambda
\end{cases} \text{\quad or simply \quad} \begin{cases}
0=-1-X^2-AY\\
\tau=\frac{AZ-XC+C\lambda}{1+\lambda^2}\\
\kappa=\frac{XZ+CY+Z\lambda}{1+\lambda^2}
\end{cases}
$$
As before, we can take a new basis, $\{e_1-\kappa e_3,e_2-\tau e_3, e_3\}$ to kill off $\tau$ and $\kappa$ (this does not compromise the condition $X+B=0$, where these are now coefficients in the new basis, either by straightforward computation or by noting that the change of base cannot interfere with $im J^*$). The resulting equations
$$
\begin{cases}
0=AZ-XC+C\lambda\\
0=XZ+CY+Z\lambda
\end{cases}
$$
quickly give $Z=C=0$.
\end{enumerate}
We summarise this in the following
\begin{proposition}
Each integrable complex structure $\J$ is (in some standard basis) of one of the two following forms. Either it is
$$
\left[
\begin{array}{cccccc}
X & \frac{-1-X^2}{Y} & 0 & 0 & 0 & 0\\
Y & -X & 0 & 0 & 0 & 0\\
0 & 0 & \lambda & 0 & 0 & -1-\lambda^2\\
0 & 0 & 0 & X^* & \frac{-1-(X^*)^2}{Y^*} & 0\\
0 & 0 & 0 & Y^* & -X^* & 0\\
0 & 0 & 1 & 0 & 0 & -\lambda
\end{array}
\right]
$$
for some numbers $\lambda$, $X$, $X^*$, $Y$, $Y^*$, or
$$
\left[
\begin{array}{cccccc}
X & A & 0 & -1-X^2-AY & -A & 0\\
Y & B & 0 & - Y(X+B) & \frac{-1-X^2-AY}{X+B} & 0\\
Z & C & \lambda & -XZ-CY-Z\lambda & \frac{-AZ-BC-C\lambda}{X+B} & -1-\lambda^2\\
1 & 0 & 0 & -X  & \frac{-A}{X+B} & 0\\
0 & (X+B) & 0 & -Y(X+B) & -B & 0\\
0 & 0 & 1 & -Z & \frac{-C}{X+B} & -\lambda
\end{array}
\right]
$$
for some $X$, $Y$, $Z$, $A$, $B$, and $C$ (under assumption that $X+B\neq 0$) but with dependent $\lambda=\frac{-1+XB-AY}{X+B}$. Any such choice of numbers provides an integrable complex structure.
\end{proposition}
\begin{remark}
Observe that the first complex structure has invariant subspaces in both $\g_3$ and $\g^*_3$ and the second has none in either. There can be no "mixed" situation, where $\g_3$ has an invariant subspace and $\g_3^*$ does not (or vice versa), since it follows from the discussion that the dimension of $J^*\g_3$ is equal to the dimension of $J\g^*_3$. That special behaviour will be possible only in algebra (6).
\end{remark}
\begin{remark}
Observe that, while we have many degrees of freedom in the second case, we cannot arrange for $J\g_3$ and $J^*\g^*_3$ to be trivial subspaces -- indeed, we will see (cf. Remark \ref{switch}) that the only algebra that allows an integrable complex structure to exchange its factors is the abelian algebra (1).
\end{remark}
\begin{remark}
We can however arrange for other quasi-invariant vectors apart from $e_3$, unlike in the previous and following algebras.
\end{remark}
\begin{remark}
We note that this case was also treated in \cite{mag2}. While further below we omit details about algebras $\g_{7}$ and $\g_{8}$ simply citing Magnin's work, Heisenberg algebra $\g_{3}$ has sufficiently different $Aut(\g\times\g)$ (which is the group acting in \cite{mag2}) from $Aut(\g)\times Aut(\g)$ (acting above) to merit explicit treatment here.
\end{remark}

\section{Algebra (4) with $\theta=1$}

The algebra is given by a multiplication table $[e_1,e_3]= e_1$, $[e_2,e_3]= e_2$, $[e_1,e_2]= 0$. Suppose that $\J$ is an integrable complex structure and
\begin{align*}
\J e_1 &=  Xe_1 + Ye_2 + Ze_3 + J^*e_1 \\
\J e_2 &=  Ae_1 + Be_2 + Ce_3 + J^*e_2 \\
\J e_3 &=  Pe_1 + Qe_2 + Re_3 + J^*e_3
\end{align*}
Then the Nijenhuis brackets are
\begin{align*}
N_{\J}(e_1,e_3) &=  [e_1,e_3] + \J[Je_1,e_3]+\J[e_1,Je_3] - [\J e_1,\J e_3] \\
&=e_1+X\J e_1+Y\J e_2+R\J e_1-XRe_1-YRe_2+ZPe_1+ZQe_2-[J^*e_1,J^*e_3]\\
N_{\J}(e_2,e_3) &=  [e_2,e_3] + \J[Je_2,e_3]+\J[e_2,Je_3] - [\J e_2,\J e_3] \\
&=e_2+A\J e_1+B\J e_2+R\J e_2-ARe_1-BRe_2+CPe_1+CQe_2-[J^*e_2,J^*e_3]\\
N_{\J}(e_1,e_2) &=  [e_1,e_2] + \J[Je_1,e_2]+\J[e_1,Je_2] - [\J e_1,\J e_2] \\
&=0-Z\J e_2+C\J e_1-XCe_1-YCe_2+AZe_1+BZe_2-[J^*e_1,J^*e_2]\\
\end{align*}
and together, after expanding all $\J$ leave us with
$$
\begin{cases}
0=1+X^2+AY+ZP\\
0=XY+YB+ZQ\\
0=XZ+YC+RZ\\
[J^*e_1,J^*e_3]=(X+R)J^*e_1+YJ^*e_2\\
0=AX+AB+CP\\
0=1+B^2+AY+CQ\\
0=AZ+BC+RC\\
[J^*e_2,J^*e_3]=AJ^*e_1+(B+R)J^*e_2\\
[J^*e_1,J^*e_2]=C J^*e_1-ZJ^*e_2
\end{cases}
$$
The possible candidates for a quasi-invariant vector are of two types.
\begin{proposition}
Each automorphism of algebra (4) must leave the subspace $span\{e_1,e_2\}$ invariant.
Each vector is equivalent to either $e_1$ or $e_3$ under some automorphism. 
\end{proposition}
It is again easy to see that some vector equivalent to $e_3$ must be quasi invariant, since the first equation prohibits any vector equivalent $e_1$ from being one. We note that unlike in the previous cases $e_3$ is not fixed by the automorphisms of $\g_4^1$, so we pick an appropriate vector and extend it to a standard basis, thus assuming without loss of generality that $P=Q=0$, and $R=\lambda$. The equations now read
$$
\begin{cases}
0=1+X^2+AY\\
0=Y(X+B)\\
0=XZ+YC+\lambda Z\\
[J^*e_1,J^*e_3]=(X+\lambda)J^*e_1+YJ^*e_2\\
0=A(X+B)\\
0=1+B^2+AY\\
0=AZ+BC+\lambda C\\
[J^*e_2,J^*e_3]=AJ^*e_1+(B+\lambda)J^*e_2\\
[J^*e_1,J^*e_2]=C J^*e_1-ZJ^*e_2
\end{cases}
$$
and because neither $A$ nor $Y$ can be zero and thus $X=-B$, we have
$$
\begin{cases}
0=1+X^2+AY\\
C=\frac{(X+\lambda) Z}{-Y}\\
[J^*e_1,J^*e_3]=(X+\lambda)J^*e_1+YJ^*e_2\\
Z=\frac{(-X+\lambda) C}{-A}\\
[J^*e_2,J^*e_3]=AJ^*e_1+(-X+\lambda)J^*e_2\\
[J^*e_1,J^*e_2]=C J^*e_1-ZJ^*e_2
\end{cases}
$$
The two equations on $C$ and $Z$ give together
$$
C=\frac{(X+\lambda) Z}{-Y}=\frac{(X+\lambda)}{-Y}\frac{(-X+\lambda) C}{-A}=C\frac{\lambda^2-X^2}{YA}
$$
which, by the top equation, amounts to $C(\lambda^2+1)=0$ and thus $C=Z=0$.

Suppose for a moment that $J^*\g_4^1$ is 3-dimensional. The last equation would then tell us that $J^*e_1$ and $J^*e_2$ lie in $span\{e_1^*,e_2^*\}$ and so the two other brackets, $[J^*e_1,J^*e_3]$ and $[J^*e_2,J^*e_3]$ must be not only non-zero, but also proportional to $J^*e_1$ and $J^*e_2$, respectively (since every vector in $span\{e_1^*,e_2^*\}$ is an eigenvector for $[e_3,\cdot]$). But this implies $A=Y=0$, contrary to the equations above.

Thus $J^*\g_4^1$ is one-dimensional and generated by $J^* e_3$. However observe that if $J^*e_1=\kappa J^*e_3$ and $J^*e_2=\tau J^*e_3$, then
\begin{align*}
-e_1 =\J(Xe_1 + Ye_2 + \kappa J^*e_3)=&  X^2e_1 +XYe_2 +X\kappa J^*e_3\\
&+ AYe_1 -XYe_2 + Y\tau J^*e_3\\
&+(-1-\lambda^2)\kappa e_3-\lambda\kappa J^*e_3
\end{align*}
and we must have $\kappa=0$, for the sake of $(-1-\lambda^2)\kappa e_3$ being 0. The same is true for $\tau$.

This is summarised in the following

\begin{proposition}
Each integrable complex structure $\J$ is (in some standard basis) of the form
$$
\left[
\begin{array}{cccccc}
X & \frac{-1-X^2}{Y} & 0 & 0 & 0 & 0\\
Y & -X & 0 & 0 & 0 & 0\\
0 & 0 & \lambda & 0 & 0 & -1-\lambda^2\\
0 & 0 & 0 & X^* & \frac{-1-(X^*)^2}{Y^*} & 0\\
0 & 0 & 0 & Y^* & -X^* & 0\\
0 & 0 & 1 & 0 & 0 & -\lambda
\end{array}
\right]
$$
for some numbers $\lambda$, $X$, $X^*$, and non-zero $Y$ and $Y^*$. Any such choice of numbers provides an integrable complex structure.
\end{proposition}
\begin{remark}
Note that this time we did not change the base after choosing the quasi-invariant vector $e_3$.
\end{remark}
\begin{remark}
Observe that every integrable complex structure must have the same invariant subspaces in $\g_4^1$ and ${\g_4^1}^*$, namely $span\{e_1,e_2\}$ and $span\{e_1^*,e_2^*\}$, respectively. Thus there can be no two distinct quasi-invariant vectors in $\g_{4}^{1}$.
\end{remark}

\section{Algebra (6) with $\theta> 0$}

The algebra is given by a multiplication table $[e_1,e_3]= \theta e_1-e_2$, $[e_2,e_3]= e_1+\theta e_2$, $[e_1,e_2]= 0$. Suppose that $\J$ is an integrable complex structure and
\begin{align*}
\J e_1 &=  Xe_1 + Ye_2 + Ze_3 + J^*e_1 \\
\J e_2 &=  Ae_1 + Be_2 + Ce_3 + J^*e_2 \\
\J e_3 &=  Pe_1 + Qe_2 + Re_3 + J^*e_3
\end{align*}
Then the Nijenhuis brackets are
\begin{align*}
N_{\J}(e_1,e_3) &=  [e_1,e_3] + \J[Je_1,e_3]+\J[e_1,Je_3] - [\J e_1,\J e_3] \\
&=\theta e_1-e_2+X\J (\theta e_1-e_2)+Y\J (e_1+\theta e_2)+R\J(\theta e_1-e_2)\\
&-XR(\theta e_1-e_2)-YR(e_1+\theta e_2)+ZP(\theta e_1-e_2)+ZQ(e_1+\theta e_2)-[J^*e_1,J^*e_3]\\
&=(\theta-XR\theta-YR+ZP\theta+ZQ)e_1+(-1+XR-YR\theta-ZP+ZQ\theta)e_2\\
&+(X\theta+Y+R\theta)\J e_1 + (-X+Y\theta - R)\J e_2 -[J^*e_1,J^*e_3]\\
N_{\J}(e_2,e_3) &=  [e_2,e_3] + \J[Je_2,e_3]+\J[e_2,Je_3] - [\J e_2,\J e_3] \\
&=e_1+\theta e_2+A\J(\theta e_1-e_2)+B\J(e_1+\theta e_2)+R\J(e_1+\theta e_2)\\
&-AR(\theta e_1-e_2)-BR(e_1+\theta e_2)+CP(\theta e_1-e_2)+CQ(e_1+\theta e_2)-[J^*e_2,J^*e_3]\\
&=(1-AR\theta-BR+CP\theta +CQ)e_1+(\theta+AR-BR\theta-CP+CQ\theta)e_2\\
&+ (A\theta+B+R)\J e_1+(-A+B\theta+R\theta)\J e_2 -[J^*e_2,J^*e_3]\\
N_{\J}(e_1,e_2) &=  [e_1,e_2] + \J[Je_1,e_2]+\J[e_1,Je_2] - [\J e_1,\J e_2] \\
&=0-Z\J(e_1+ \theta e_2)+C\J(\theta e_1- e_2)\\
&-XC(\theta e_1-e_2)-YC(e_1+\theta e_2)+AZ(\theta e_1-e_2)+BZ(e_1+\theta e_2)-[J^*e_1,J^*e_2]\\
&=(-Z+C\theta)\J e_1 +(-Z\theta -C)\J e_2\\
&+(-XC\theta - YC+AZ\theta+BZ)e_1+(XC-YC\theta-AZ+BZ\theta)e_2-[J^*e_1,J^*e_2]
\end{align*}
and together, after expanding all $\J$'s leave us with
\begin{align*}
(X\theta + Y)(Xe_1 + Ye_2 + Ze_3 + J^*e_1)+(-X+Y\theta -R)(Ae_1 + Be_2 + Ce_3 + J^*e_2)&\\
\qquad+(\theta-YR+ZP\theta+ZQ)e_1+(-1+XR-ZP+ZQ\theta)e_2+Z\theta R e_3 -[J^*e_1,J^*e_3]=&0\\
(A\theta+B+R)(Xe_1 + Ye_2 + Ze_3 + J^*e_1)+ (-A+B\theta)(Ae_1 + Be_2 + Ce_3 + J^*e_2)&\\
\qquad+(1-BR+CP\theta +CQ)e_1+(\theta+AR-CP+CQ\theta)e_2 +CR\theta e_3 -[J^*e_2,J^*e_3]=&0\\
-Z(Xe_1 + Ye_2 + Ze_3 + J^*e_1) -C(Ae_1 + Be_2 + Ce_3 + J^*e_2)&\\
\qquad+(- YC+BZ)e_1+(XC-AZ)e_2-[J^*e_1,J^*e_2]=&0
\end{align*}
Note that the $e_3$-part of the last equation is $-Z^2-C^2$, and so implies that both $Z$ and $C$ are zero. We rewrite all the equations using this information.
$$
\begin{cases}
0=(X\theta + Y)X+(-X+Y\theta -R)A+\theta-YR\\

0=(X\theta + Y)Y+(-X+Y\theta -R)B-1+XR\\

[J^*e_1,J^*e_3]=(X\theta + Y)J^*e_1+(-X+Y\theta -R)J^*e_2\\

0=(A\theta+B+R)X+ (-A+B\theta)A+1-BR\\

0=(A\theta+B+R)Y+ (-A+B\theta)B+\theta+AR\\

[J^*e_2,J^*e_3]=(A\theta+B+R)J^*e_1+ (-A+B\theta)J^*e_2\\

[J^*e_1,J^*e_2]=0
\end{cases}
$$
The possible candidates for a quasi-invariant vector are of two types.
\begin{proposition}
Each automorphism of algebra (6) must leave the subspace $span\{e_1,e_2\}$ invariant. Each vector is equivalent to either $e_1$ or $e_3$ under some automorphism. 
\end{proposition}
\begin{proposition}
No vector equivalent to $e_1$ is quasi-invariant.
\end{proposition}
\begin{proof}
Assume to the contrary. Then $Y=Z=0$ and $X=\lambda$ and the equations become
$$
\begin{cases}
0=\theta(\lambda^2+1)+A(-\lambda-R)\\
0=B(-\lambda-R)-1+\lambda R\\
[J^*e_1,J^*e_3]=\lambda\theta J^*e_1+ (-\lambda-R)J^*e_2\\
0=(A\theta+B+R)\lambda+(-A+B\theta)A +1-BR\\
0=(-A+B\theta)B+\theta+AR\\
[J^*e_2,J^*e_3]=(A\theta+B+R)J^*e_1 + (-A+B\theta)J^*e_2\\
[J^*e_1,J^*e_2]=0
\end{cases}
$$
We can rewrite the first two and the fifth equation as
$$
\begin{cases}
\theta=\frac{A(\lambda+R)}{1+\lambda^2}\\
\lambda=\frac{1+BR}{R-B}\\
\theta=\frac{A(B-R)}{1+B^2}
\end{cases}
$$
where we divided by $R-B$ because $\theta\neq 0$ tells us that none of $A$, $\lambda+R$, or $B-R$ can be zero. We thus get
$$
\frac{A(\frac{1+BR}{R-B}+R)}{1+(\frac{1+BR}{R-B})^2}=\frac{A(B-R)}{1+B^2}
$$
that gives
$$
\frac{(1+R^2)(R-B)}{(1+R^2)(1+B^2)}=\frac{B-R}{1+B^2}
$$
or $R-B=B-R$, which is impossible.
\end{proof}

Therefore we can assume that $e_3$ is invariant (again, extending the quasi-invariant vector to a standard basis), or $P=Q=0$ and $R=\lambda$. The equations become
$$
\begin{cases}
0=(X\theta + Y)X+(-X+Y\theta -\lambda)A+\theta-Y\lambda\\

0=(X\theta + Y)Y+(-X+Y\theta -\lambda)B-1+X\lambda=&0\\

[J^*e_1,J^*e_3]=(X\theta + Y)J^*e_1+(-X+Y\theta -\lambda)J^*e_2\\

0=(A\theta+B+\lambda)X+ (-A+B\theta)A+1-B\lambda=&0\\

0=(A\theta+B+\lambda)Y+ (-A+B\theta)B+\theta+A\lambda=&0\\

[J^*e_2,J^*e_3]=(A\theta+B+\lambda)J^*e_1+ (-A+B\theta)J^*e_2\\

[J^*e_1,J^*e_2]=0
\end{cases}
$$
As in the case of $\g_{4}^1$, we can see that $Z=C=0$ implies that if $J^*\g_{6}^{\theta}$ is 1-dimensional, then $J^*e_1=J^*e_2=0$. We will now show that this is almost always the case.

Suppose that $J^*\g_{6}^{\theta}$ is 3-dimensional. Then the equations above allow for a computation of the characteristic polynomial for $[\cdot,J^*e_3]$
\begin{align*}
-t\left((X\theta+Y-t)\right.&\left.(-A+B\theta-t)-(-X+Y\theta-\lambda)(A\theta+B+\lambda)\right)=\\
&-t\left(t^2+t(A-B\theta-X\theta-Y)+(X\theta+Y)(-A+B\theta)\right.\\
&\left.-(-X+Y\theta-\lambda)(A\theta+B+\lambda)\right)
\end{align*}
and we know the roots of the quadratic function: $\theta\pm i$ (easily computed in the standard basis). This gives us two additional (Vieta's) equations
$$
\begin{cases}
A-B\theta-X\theta-Y=-2\theta\\
4\theta^2-4(X\theta+Y)(-A+B\theta)+4(-X+Y\theta-\lambda)(A\theta+B+\lambda)=-4
\end{cases}
$$
which we rewrite along all the others
\begin{numcases}{}
A-Y=\theta(X+B-2) \\
\theta^2(1-XB+AY)=-1+XB-AY+\theta\lambda (A-Y)+\lambda^2\\
\theta(1+X^2+AY)=X(A-Y)+\lambda(A+Y)\\
\theta A(X+B)=-1+A^2-BX+\lambda(B-X)\\
\theta Y(X+B)=1-Y^2+BX+\lambda(B-X)\\
\theta(1+B^2+AY)=B(A-Y)-\lambda(A+Y)
\end{numcases}
Equations (3)-(6) and (4)+(5) give, after substituting (1)
$$
\begin{cases}
\theta(X-B)(X+B)=(X-B)\theta(X+B-2)+2\lambda(A+Y)\\
\theta(A+Y)(X+B)=\theta(X+B-2)(A+Y)+2\lambda(B-X)
\end{cases}
$$
or
$$
\begin{cases}
\theta(X-B)-\lambda(A+Y)=0\\
\lambda(X-B)+\theta(A+Y)=0
\end{cases}
$$
of which the determinant $\theta^2+\lambda^2$ is non-zero, and hence $X-B=0$ and $A+Y=0$. This leaves us with
$$
\begin{cases}
Y=\theta(1-X)\\
\theta(1+X^2-Y^2)=-2XY\\
2\theta XY=1+X^2-Y^2\\
\theta^2(1-X^2-Y^2)=-1+X^2+Y^2-2\theta\lambda Y+\lambda^2+2\lambda X
\end{cases}
$$
The first equation tells us that $Y$ is non-zero, and the following two give $2\theta^2 XY=-2XY$, which means that $X=0$. We are left with
$$
\begin{cases}
Y=\theta\\
1-Y^2=0\\
\lambda(\lambda-2)=0
\end{cases}
$$
We summarise the discussion in the following.

\begin{proposition}
Each integrable complex structure $\J$ on $\g_{6}^{\theta}\times\g_{6}^{\theta}$ is (in some standard basis) of one of the four distinct forms:
\begin{enumerate}
\item If $J^*\g_{6}^{\theta}$ and $J{\g_{6}^{\theta}}^*$ are 1-dimensional, then $\J$ is
$$
\left[
\begin{array}{cccccc}
X & \frac{-1-X^2}{Y} & 0 & 0 & 0 & 0\\
Y & -X & 0 & 0 & 0 & 0\\
0 & 0 & \lambda & 0 & 0 & -1-\lambda^2\\
0 & 0 & 0 & X^* & \frac{-1-(X^*)^2}{Y^*} & 0\\
0 & 0 & 0 & Y^* & -X^* & 0\\
0 & 0 & 1 & 0 & 0 & -\lambda
\end{array}
\right]
$$
for some numbers $\lambda$, $X$, $X^*$, and non-zero $Y$ and $Y^*$. Any such choice of numbers provides an integrable complex structure. This can occur for any $\theta$.
\item If $\theta=1$, $J^*\g_{6}^{\theta}$ or $J{\g_{6}^{\theta}}^*$ can be 3-dimensional, and we then have three possibilities:
\begin{itemize}
\item If $\lambda$ is $0$, then $\J$ is
$$
\left[
\begin{array}{cccccc}
0 & -1 & 0 & 0 & 0 & 0\\
1 & 0 & 0 & 0 & 0 & 0\\
0 & 0 & 0 & 0 & 0 & -1\\
0 & 1 & 0 & 0 & -1 & 0\\
1 & 0 & 0 & 1 & 0 & 0\\
0 & 0 & 1 & 0 & 0 & 0
\end{array}\right]\text{\quad or \quad}
\left[\begin{array}{cccccc}
0 & -1 & 0 & 0 & 1 & 0\\
1 & 0 & 0 & 1 & 0 & 0\\
0 & 0 & 0 & 0 & 0 & -1\\
0 & 0 & 0 & 0 & -1 & 0\\
0 & 0 & 0 & 1 & 0 & 0\\
0 & 0 & 1 & 0 & 0 & 0
\end{array}
\right]
$$
Note that $J^*e_1=e_2^*$ (or, in the second case, $Je_1^*=e_2$), forced by the Nijenhuis tensor and equations on the brackets in $J^*\g_{6}^{\theta}$).
\item If $\lambda$ is $2$, then $\J$ is
$$
\left[
\begin{array}{cccccc}
0 & -1 & 0 & 0 & 0 & 0\\
1 & 0 & 0 & 0 & 0 & 0\\
0 & 0 & 2 & 0 & 0 & -5\\
1 & 0 & 0 & 0 & 1 & 0\\
0 & 1 & 0 & -1 & 0 & 0\\
0 & 0 & 1 & 0 & 0 & -2
\end{array}
\right]
$$
Note that this time $J^*e_1=e_1^*$.
\item If $\lambda$ is $-2$, then $\J$ is
$$
\left[
\begin{array}{cccccc}
0 & 1 & 0 & 1 & 0 & 0\\
-1 & 0 & 0 & 0 & 1 & 0\\
0 & 0 & -2 & 0 & 0 & 1\\
0 & 0 & 0 & 0 & -1 & 0\\
0 & 0 & 0 & 1 & 0 & 0\\
0 & 0 & -5 & 0 & 0 & 2
\end{array}
\right]
$$
\end{itemize}
\end{enumerate}
\end{proposition}

\section{Algebras (7) and (8) -- $\mathfrak{su}(3)$ and $\mathfrak{sl}(2,\mathbb{R})$}

For completeness, we include the results of Magnin:

\begin{proposition}{(\cite{mag1}, Corollary 3)}
Each integrable complex structure on $\mathfrak{su}(3)\times\mathfrak{su}(3)$ is (in some standard basis) of the form
$$
\left[
\begin{array}{cccccc}
0 & -1 & 0 & 0 & 0 & 0\\
1 & 0 & 0 & 0 & 0 & 0\\
0 & 0 & \lambda & 0 & 0 & \eta \\
0 & 0 & 0 & 0 & -1 & 0\\
0 & 0 & 0 & 1 & 0 & 0\\
0 & 0 & \frac{-1-\lambda^2}{\eta} & 0 & 0 & -\lambda
\end{array}
\right]
$$
for some choice of $\lambda$ and non-zero $\eta$.
\end{proposition}

\begin{proposition}{(\cite{mag2}, Proposition 4.1)}
Each integrable complex structure on $\mathfrak{sl}(2,\mathbb{R})\times\mathfrak{sl}(2,\mathbb{R})$ is (in some standard basis) of the form
$$
\left[
\begin{array}{cccccc}
0 & -1 & 0 & 0 & 0 & 0\\
1 & 0 & 0 & 0 & 0 & 0\\
0 & 0 & \lambda & 0 & 0 & \eta \\
0 & 0 & 0 & 0 & -1 & 0\\
0 & 0 & 0 & 1 & 0 & 0\\
0 & 0 & \frac{-1-\lambda^2}{\eta} & 0 & 0 & -\lambda
\end{array}
\right]
$$
for some choice of $\lambda$ and non-zero $\eta$.
\end{proposition}
We note that Magnin's work contains a deep study of the moduli space of all integrable complex structures on the relevant square algebras under the action of the whole $Aut(\g\times\g)$ (which does happen in these cases to be not very different from $Aut(\g)\times Aut(\g)$, namely $Aut(\g\times\g)=Aut(\g)\times Aut(\g)\cup \tau(Aut(\g)\times Aut(\g))$ where $\tau$ is the switch between factors of $\g\times\g$). 

\section{Concluding observations}

\begin{remark}
We point out that $\g_{6}^{1}$ is the only algebra for which a complex structure can exhibit different behaviours on both parts of $\g\times\g$.
\end{remark}

\begin{remark} \label{switch}
We see that the only algebra that allows a complex structure that switches its summands ($\J\g=\g^*$ and $\J\g^*=\g$) is the abelian algebra.
\end{remark}

\begin{remark} \label{mix}
As we mentioned before, Theorem 1 and Proposition 2 in \cite{jamarcin} state that each of the algebras admitting an integrable complex structure admits one of a very special type: $\J u=u^*$, $\J v=w$, $\J v^*=w^*$, where $[u,\cdot]$ and $[u^*,\cdot]$ have a complex eigenvalue $A+Bi$ (with $B\neq 0$) or at least double real eigenvalue, for which $v$ and $w$, and $v^*$ and $w^*$ are eigenvectors.  All six vectors constitute a standard basis.
\end{remark}

\begin{remark}
While the present paper and \cite{jamarcin} depend on the low dimension of the algebras in question, and especially on the existence of a classification that allows for a case-by-case study, the previous Remark indicates that one can expect progress in higher dimensions dependent only on possible Jordan forms of the adjoints $[v,\cdot]$, without full classification or description of multiplication tables.
\end{remark}

\end{document}